\documentclass[11pt]{amsart}

\usepackage{amsthm}
\usepackage{amsmath}
\usepackage{amssymb}

\newtheorem{Thm}{Theorem}[section]

\newtheorem{lemma}[Thm]{Lemma}

\newtheorem{proposition}[Thm]{Proposition}
\newtheorem{definition}[Thm]{Definition}
\newtheorem{remark}[Thm]{Remark}

\newtheorem{example}[Thm]{Example}
\newtheorem{theorem}[Thm]{Theorem}

\def\ldots{\mathinner{\ldotp\ldotp\ldotp}}
\def\ldots{\mathinner{\cdotp\cdotp\cdotp}}

\def \H{\mathbb H}
\def\F{\mathcal F}
\def\E{\cal E}
\def\G{\cal G}
\def \N{\mathbb N}
\def\Z{\mathbb Z}
\def\R{\mathbb R}

\def \cal{\mathcal}

\def \beq{\begin{eqnarray*}}
\def \eeq{\end{eqnarray*}}

\def \cc{\mathbf c}
\def \uu{ u }

\title[Infinite dimensional restricted invertibility]{Infinite dimensional
restricted invertibility}

\author[P.G. Casazza and G.E. Pfander]
{\small Peter G. Casazza and G\"otz E. Pfander}
\date{\small\today}

\email{pete@math.missouri.edu; g.pfander@jacobs-university.de}

\thanks{The first author was supported by NSF DMS 0704216 and
thanks the American Institute of Math for their continued support.}

\begin{document}

\maketitle

\begin{abstract}
The 1987 {\em Bourgain-Tzafriri Restricted Invertibility Theorem} is one of the most
celebrated theorems in analysis.  At the time of their work, the authors
raised the
question of a possible infinite dimensional version of the theorem.
  In this paper, we will  give a quite general definition of restricted
invertibility for operators on
infinite dimensional Hilbert spaces  based on the notion of  {\em density}
from frame theory.   We then prove that localized
Bessel systems have large subsets which are Riesz basic sequences.
As a consequence, we prove the strongest possible form of the infinite
dimensional
restricted invertibility theorem for $\ell_1$-localized operators
and for Gabor frames with generating function in the Feichtinger
Algebra.  For our calculations, we introduce a new notion of {\em density}
which has serious advantages over the standard form because it
is independent of index maps - and hence has much broader
application.  We then show that in the setting of the
restricted invertibility theorem, this new density becomes equivalent to
the standard density. \\

\end{abstract}

\section{Introduction}\label{Section:Introduction}

In 1987, Bourgain and Tzafriri proved one of the most celebrated and
useful theorems in analysis \cite{BT}:  {\em The Bourgain-Tzafriri Restricted
Invertibility Theorem}.  The form we give now can be found in Casazza \cite{C},
Vershynin \cite{V1} (where the restriction that the norms of the vectors
$Te_i$ equal one - or even are bounded below - is removed), and
Vershynin \cite{V2,V3} (also see Casazza and Tremain \cite{CT1}).

\begin{theorem}[\bf Restricted Invertibility Theorem]
\label{theorem:RestrictedInvertTheoremFinite1}
There exists a function
$\cc:(0,1)\longrightarrow (0,1)$ so that for every $n\in \N$ and every linear
operator $T:\ell_2^n \rightarrow \ell_2^n$ with $\|Te_i\|=1$ for
$i=1,2,\ldots,n$ and $\{e_i\}_{i=1}^n$ an orthonormal basis for $\ell_2^n$,
there is a subset $J_\epsilon\subseteq \{1,2,\ldots,n\}$ satisfying\\[.1cm]

\noindent (1)  \quad $\displaystyle \frac {|J_\epsilon|} n \ge
\frac{(1-\epsilon)}{\|T\|^2}, $

\vskip12pt

\noindent (2)  For all $\{b_j\}_{j\in J_{\epsilon}}\in \ell_2(J_\epsilon)$
we have
\[
\|\sum_{j\in J_\epsilon}b_jTe_j\|^2 \ge \cc(\epsilon)\,
\sum_{j\in J_\epsilon}|b_j|^2.\]
\end{theorem}

Throughout this paper, $\|\cdot\|$ represents the Hilbert space norm on vectors and the operator norm for operators acting on Hilbert spaces.

In our proofs we will need a minor extension of Theorem
\ref{theorem:RestrictedInvertTheoremFinite1} which is stated and
proved in the appendix (see Theorem~\ref{theorem:RestrictedInvertTheoremFinite}).
It is easily seen that (1) is best possible in Theorem
\ref{theorem:RestrictedInvertTheoremFinite1}.  Letting $Te_{2i}=e_i = Te_{2i-1}$
for $i=1,2,\ldots, n$ in $\ell_2^{2n}$, we see that $1/\|T\|^2$ is necessary.
In \cite{CP} it is shown that the class of equal norm Parseval frames
$\{f_i\}_{i=1}^{2n}$ in $\ell_2^n$ are not {\em 2-pavable}.  In the
current setting, this says that Theorem
\ref{theorem:RestrictedInvertTheoremFinite1} (1) fails if $\epsilon = 0$.

In their paper \cite{BT}, Bourgain and Tzafriri raised the question of
a possible infinite dimensional version of their theorem.  They then
gave a weakened version of this for the special
case of families of exponentials.  Vershynin \cite{V3} proves an
infinite dimensional restricted invertibility theorem for restrictions
of exponentials to subsets of the torus.

In this paper, we will use the notion of {\em density} from frame theory
to give a precise definition for infinite dimensional restricted invertibility.
We then prove a very general theorem on restricted invertibility for
classes of Bessel systems which are $\ell_1$-localized with respect to
frames.  As a consequence, we obtain the general restricted invertibility
theorem for $\ell_1$-localized operators on arbitrary Hilbert spaces.
We apply our general results to  prove the restricted
invertibility theorem for Gabor systems with generator in the Feichtinger
algebra as well as for systems of Gabor molecules in the Feichtinger algebra.

Standard density theory requires an {\em index map} (see Section
\ref{Section:Notation}.) This can be problematic in some applications.
So we will introduce a new notion of density which is independent of
index maps and as a consequence should have much broader application
in the field.  We will then show that in the presence of {\em localization},
this form of density becomes equivalent to the standard form.

The notion of localization with respect to an orthonormal basis is not
usable in Gabor theory due to the Balian-Low Theorem \cite{G}.
This is why we have to move from {\em rectangular} coordinate systems
to {\em overcomplete} coordinate systems.  This leads us to introduce
a new concept of {\em relative density}, because there, the overcompleteness
of the coordinate system factors out.

The paper is organized as follows.  Section 2 contains the notation, the
first form of {\em density} and the statements of the fundamental results
in the paper.  Section 3 is a detailed discussion of {\em localization}
with a number of examples.  Here, we also introduce our new notion of
{\em density} which has the major advantage that it is independent
of index maps.  We then show its relationship to the standard {\em density}
and show that in the setting of $\ell_2$-localized frames, the two forms
of density are the same.  We also
restate our main results using the second notion of density.  Section 4 contains
the proof of the main results on restricted invertibility.  Section 5
addresses  the  restricted invertibility theorem for Gabor
systems and Section 6 is an
appendix containing some intermediate results used in this paper.

\section{Notation and statement of results}\label{Section:Notation}

Hilbert space frame theory has traditionally been used in
signal processing (see \cite{G}) but recently has also had a significant impact
on problems in pure mathematics, applied mathematics and engineering.
(See, for example, \cite{C1,CFTW,CT,Ch,Pfa07b} and their references.)

\begin{definition}\label{definitionframe}
A family of vectors $\{f_i\}_{i\in I}$ in a Hilbert space $\H$ is called
a {\bf frame} for $\H$ if there are constants $0<A\le B<\infty$
(called {\bf lower and upper frame bounds} respectively) if
\[
A\|f\|^2 \le \sum_{i\in I}|\langle f,f_i\rangle|^2 \le B \|f\|^2,
\ \ \mbox{for all $f\in \H$}.\]
\end{definition}
If we only have the right hand side inequality, we call the family
a {\bf Bessel sequence} with {\bf Bessel bound} $B$.
If we can choose $A=B$ in Definition \ref{definitionframe}, then we say the frame is
{\bf tight} with tight frame bound $A$.  If $A=B=1$, it is a
{\bf Parseval frame}.  The {\bf analysis operator}
$T:\H \rightarrow \ell_2(I)$ of the
frame $\{f_i\}_{i\in I}$ is defined by
\[ T(f) = \sum_{i\in I} \langle f,f_i\rangle e_i,\]
where $\{e_i\}_{i\in I}$ is the unit vector basis of $\ell_2(I)$.
The adjoint of $T$ is the {\bf synthesis operator} given by
\[ T^{*}(e_i) = f_i, \ \ \mbox{for all $i\in I$}.\]
The {\bf frame operator} is the positive, self-adjoint, invertible
operator $S:\H \rightarrow \H$ where $S=T^{*}T$.  That is, for
all $f\in \H$,
\[ S(f) = \sum_{i\in I}\langle f,f_i\rangle f_i.\]
{\bf Reconstruction} of $f\in \H$ comes from
\[ f = \sum_{i\in I}\langle f,f_i\rangle S^{-1}(f_i).\]
The family $\{S^{-1}(f_i)\}_{i\in I}$ is also a frame for $\H$ called
the {\bf dual frame} of $\{f_i\}_{i\in I}$.

A family of vectors $\{f_i\}_{i\in I}$ in $\H$ is called a {\bf Riesz
sequence} with {\bf Riesz bounds} $0<A\le B <\infty$ if for all families of
scalars $\{a_i\}_{i\in I}$ we have
\[ A \sum_{i\in I}|a_i|^2 \le \|\sum_{i\in I}a_if_i\|^2 \le
B \sum_{i\in I}|a_i|^2.\]

We will use the notion of {\em density} from frame theory to
give the correct formulation of restricted invertibility for
infinite dimensional Hilbert spaces.  In the following section we will
define the previously mentioned new notion of density which does not require an index map
and then show that for $\ell_2$-localized frames, the two notions
of density are equivalent,  a result which is interesting in itself.

Over the last few years, a
considerable amount of work has been done on {\em density theory}.
We refer the reader to
\cite{BCHZ, BCHZ1, BCHZ2, BCZ} for the latest developments.
The common notions on density involve countable point sets in $\sigma$-finite
discrete measure spaces. We follow this approach and, throughout the paper,
$I$ will denote a countable index set and $G$ will
denote a finitely generated Abelian group
$G=\Z^{d_1}\times \Z_{N_1}\times \cdot \times  \Z_{N_{d_2}}$ with $d_1,d_2\in\N$ and $\Z_N=\{0,1,2,\ldots,N-1\}$ being the
cyclic group of order $N$.

\begin{definition}\label{definition:density}
Let $I$ be a set and
$a:I\longrightarrow G$ (called a {\bf localization map}).
 For $J\subseteq I$, the {\bf lower and upper density} of $J$
with respect to $a$ are given, respectively, by
   \begin{eqnarray}
     D^-(a;J)&=&\liminf_{R\to\infty} \inf_{k\in G} \frac{|a^{-1} (B_R(k) )
\cap J|}{ |B_R(0)|} , \label{equation:lowerdensity}\\
     D^+(a;J)&=&\limsup_{R\to\infty} \sup_{k\in G} \frac{|a^{-1} (B_R(k) )
\cap J|}{ |B_R(0)|}, \label{equation:lowerdensity}
   \end{eqnarray}
where $|\cdot |$ denotes the cardinality of the set and
\[ B_R(k)=\{g\in G\,;\,\|g-k\|_{\infty} = \max_{1\le j\le d_1+d_2}|g(j)-k(j)|\leq R\} \]
 is the box of radius $R$ and
center $k$ in $G$.
Note that $|B_R(k)|=|B_R(k')|$ for all $k, k'\in G$ and $R>0$.

If $D^-(a;J)=D^+(a;J)$, then we say that $J$ is of {\bf uniform density}
and write $D(a;J)=D^-(a;J)=D^+(a;J)$.
\end{definition}

\begin{remark} \rm In the case that $I=G$ and $a= id$, we write the lower
and upper density as $D^{-}(J), D^{+}(J)$ and these are called the
{\bf Beurling densities} of $J$ \cite{BCHZ,BCHZ1,BCHZ2,G}.
\end{remark}

The dependence of  $D^-(a;J)$ and $D^-(a;J)/D^-(a;I)$
 on $a$  is
illustrated in the following example.
\begin{example}\label{example:densitydependsona}\rm
Let $I=G=\mathbb Z$ and $J=2\mathbb Z$.
\begin{enumerate}
  \item For $a={\rm id}$, we have $D^-(a;J)/D^-(a;I)=\frac {1/2} 1=\frac 1 2$.
  \item For $a=2{\rm id}$ we have $D^-(a;J)/D^-(a;I)=\frac {1/4}
{1/2}= \frac 1 2$.
  \item For $a$ bijective with even numbers mapping bijectively to
$\mathbb Z\setminus 4\mathbb Z$ and odd numbers to $4\mathbb Z$, we have
$D^-(a;J)/D^-(a;I)=\frac {3/4} {1}=\frac 3 4$.
\end{enumerate}
\end{example}

Nonetheless, this dependence on $a$ will not introduce ambiguity when combined
with standard localization notions from frame theory
(see, for example, \cite{BCHZ, BCHZ1, BCHZ2, BCZ}).

\begin{definition}\label{definition:localization}Let $p=1$ or $p=2$. Let $a:I\longrightarrow G$,
 and let $\G = \{g_k: k\in G\}$ be a frame for $\H$ and $\F=\{f_i\}_{i\in I}
\subseteq \H$.  We say that $(\F,a,\G)$ is $\ell_p${\bf -localized}  if there
exists
   $r\in \ell_p(G)$ with $ |\langle f_i, g_{k'}\rangle |\leq r(k)$ whenever
$a(i)-k'=k$. Also, $\G = \{g_k: k\in G\}$ is $\ell_p${\bf -self-localized} if
$(\G,{\rm id},\G)$ is $\ell_p$-localized.

The operator $T:\H'\longrightarrow \H$ is $\ell_p$-{\bf localized} if there
exists an orthonormal basis $\E$ of $\H'$ indexed by $I$, a frame $\G$ of $\H$ indexed by the finitely generated Abelian group $G$, and a map $a:I\longrightarrow G$ so that
 so that $(T(\E),a,\G)$ is $\ell_p$-localized.
\end{definition}

As discussed in detail in Section~\ref{section:localization}, given $\F$ and
 $\G$, $D^-(a;J)$ and $D^+(a;J)$ do not depend on the choice of $a$ as
long as $(\F,a,\G)$ is $\ell_2$-localized.

We can now state the main results of the paper.  The first is the
frame theoretic form of restricted invertibility.

\begin{theorem}\label{theorem:MainForRiesz} Let $\cc$ be the function
provided in Theorem~\ref{theorem:RestrictedInvertTheoremFinite}.
Let $\F=\{f_i\}_{i\in I}$, $\|f_i\|\geq \uu>0$ for all $i\in I$,  be a Bessel system with Bessel bound $B$ in a Hilbert space $\H$. Let $G$ be a finitely generated Abelian group and
assume either\\[.1cm]

\noindent (A) $\G = \{g_k: k\in  G\}$ is a Riesz basis for $\H$ with Riesz
bounds $A,B$,

or

\noindent (B) $\G = \{g_k: k\in G\}$ is a frame for $\H$ with
$\ell_1$- self-localized dual frame
$\widetilde \G = \{\widetilde g_k: k\in G\}$.\\[.1cm]

\noindent Let $a:I \rightarrow G$ be a localization map with
$
    0<D^-(a;I)\leq D^+(a;I)<\infty.
$
If $(\F,a,\G)$ is $\ell_1$-localized, then for every
$\epsilon>0$ and $\delta>0$ there is a subset $ J=J_{\epsilon\delta} \subseteq I$
of uniform density satisfying \\[.1cm]

\noindent (1)  \quad $\displaystyle
  \frac{D(a;J)}{ D^{-}(a;I)}
\ge \frac{(1-\epsilon)\uu^2}{B}$,\\[.1cm]


\noindent (2) For all scalars $\{b_j\}_{j\in J}$ we have
\[  \|\sum_{j\in J}b_j f_j\|^2 \ge \cc(\epsilon)
(1-\delta) \frac A B\, \uu^2\, \sum_{j\in J}|b_j|^2
\] 
with $A=B$ in the case of (B).
%

\end{theorem}
%

A special case of Theorem~\ref{theorem:MainForRiesz} is the
restricted invertibility theorem (as envisioned by Bourgain and Tzafriri)
for $\ell_1$-localized operators on infinite dimensional
Hilbert spaces. In fact, for an orthonormal basis $\E=\{e_i\}_{i\in I}$ in $\H'$ and a bounded operator $T:\H'\longrightarrow \H$, $\{Te_i\}_{i\in I}$ is Bessel with optimal Bessel bound $\|T\|^2$.

The reader may substitute
$\Z$ or even $\N$ for the finitely generated Abelian group $G=\Z^d\times H$, $H$ finite Abelian, in the theorem below.

\begin{theorem}[{\bf Infinite Dimensional Restricted Invertibility Theorem}]
\label{theorem:RestrictedInvertTheoremInfinite}

Let $\{e_k\}_{k\in G}$ and $\G =\{g_k\}_{k\in G}$ be orthonormal bases for a
Hilbert space $\H$,
$T:\H \rightarrow \H$ be a bounded linear operator satisfying
$\|Te_k\| = 1$ for all $k\in G$ and
$\F=T(\G)$. Let $a:G \rightarrow G$ be a one to
one map and assume that $(\F,a,\G)$ is $\ell_1$-localized.  Then for all $\epsilon,\delta >0$,
there is a subset $J=J_{\epsilon\delta}\subseteq G$ of uniform density so
that
(with $\cc$ being the function provided in
Theorem~\ref{theorem:RestrictedInvertTheoremFinite}),
\\[.1cm]

\noindent (1) $\displaystyle  D(a;J)
\ge \frac{1-\epsilon}{\|T\|^2}$, \\[.1cm]
\vskip12pt

\noindent (2) For all $\{b_j\}_{j\in J}\in \ell_2(J)$ we have
\vskip12pt
\[ \|\sum_{j\in J}b_jTe_j\|^2
\ge \cc(\epsilon) (1-\delta) \sum_{j\in J}|b_j|^2.\]
\end{theorem}

Theorem \ref{theorem:RestrictedInvertTheoremInfinite} is best possible in the
sense that the theorem fails in general if $\epsilon =0$ in (1).  This follows easily
from the corresponding finite dimensional result discussed after Theorem
\ref{theorem:RestrictedInvertTheoremFinite1}.

The density concepts outlined above were developed in part
to obtain sophisticated results on the
density of Gabor frames for $L^2(\mathbb R^d)$ \cite{BCHZ1,BCHZ2,
BCZ,G}.

For  $\lambda=(x,\omega)\in \R^{2d}$ we define {\bf modulation by }$\omega$  $M_\omega$ and
{\bf translation by} $x$  $T_x$ on $L^2(\R^d)$ by
\[ M_\omega (\varphi)(\cdot) = e^{2\pi i\omega \cdot} \varphi (\cdot) ,\ \ T_x (\varphi)(\cdot) = \varphi(\cdot -x),\quad \varphi \in L^2(\R^d)\]
For $\varphi \in L^2(\mathbb R^d)$ and $\Lambda\subseteq \mathbb R^{2d}$
discrete, we consider the set $(\varphi,\Lambda)=\{\pi(\lambda)\varphi\}_{\lambda \in \Lambda}
\subseteq  L^2(\mathbb R^d)$ where
$\pi(\lambda)\varphi=\pi(x,\omega)\varphi =M_\omega T_x \varphi$, $\lambda=(x,\omega) \in\mathbb R^{2d}.$
The set  $(\varphi,\Lambda)$ is called {\bf Gabor system} with generating function
$\varphi$, and if  $(\varphi,\Lambda)$ is a frame for
$L^2(\mathbb R^d)$, then we call
$(\varphi,\Lambda)$ a {\bf Gabor frame}

The {\bf Feichtinger algebra}
is given by
$$S_0(\mathbb R^d)= \big\{ f\in L^2(\mathbb R^d):\ \ \langle f,
\pi(\cdot) g_0\rangle \in L^1(\mathbb R^{2d})\big\},$$
with $\varphi_0$ being a Gaussian \cite{G}.
Theorem~\ref{theorem:GaborExample2} in Section~\ref{section:Gabor} is
Theorem~\ref{theorem:MainForRiesz} applied to time--frequency molecules.
In terms of  Gabor frames and the lower Beurling density
$D^-(\Lambda)$ (respectively uniform Beurling density
$D(\Lambda)$), if $\Lambda\subseteq \R^{2d}$,  it reduces
to the following result.

\begin{theorem}\label{theorem:GaborExample1} Let $\epsilon,\delta>0$. Let
$\varphi \in S_0(\R)$ and let the Gabor system $(\varphi,\Lambda)$ have Bessel
bound $B<\infty$. Then exists a set $\Lambda_{\epsilon\delta}
\subseteq \Lambda$, of uniform density, so that\\[.1cm]

\noindent (1) \quad $\displaystyle
\frac{D(\Lambda_{\epsilon\delta})}{D^-(\Lambda)}
\geq \frac{(1-\epsilon)}
B\|\varphi\|^2$, \\[.1cm]
\vskip12pt
\noindent (2)  For all $\{b_\lambda\}_{\lambda
\in \Lambda}\in \ell_2(\Lambda)$,
\vskip12pt
\[ \|\sum_{\lambda \in \Lambda_{\epsilon\delta}}
b_\lambda \pi (\lambda) \varphi \|^2
\ge \cc(\epsilon) (1-\delta) \|\varphi\|
\sum_{\lambda \in \Lambda_{\epsilon\delta}}|b_\lambda|^2\]
\end{theorem}

Note that Theorem~\ref{theorem:GaborExample1} (2) states
that $(\varphi,\Lambda_{\epsilon\delta})$  is a Riesz sequence with lower
Riesz bound $\cc(\epsilon)(1-\delta)\|\varphi\|$.  That is, the lower
Riesz bound of $(\varphi,\Lambda_{\epsilon\delta})$ depends only on
$\epsilon$, $\delta$, and $\|g\|$, but not on any geometric properties
of $\Lambda$ or other specifics of $g$. Certainly, such properties of $g$ and $\Lambda$ affect the Bessel bound of $(\varphi,\Lambda)$ and therefore (1) in   Theorem~\ref{theorem:GaborExample1}.
Moreover, note that if $(\varphi,\Lambda)$ is a tight frame,
then $D^-(\Lambda) =\frac{ B}{\|\varphi\|^2}$, and (1)
in Theorem~\ref{theorem:GaborExample1} becomes
simply \cite{BCHZ2} $$\displaystyle
 D(\Lambda_{\epsilon\delta})\geq (1-\epsilon).$$

 Balan, Casazza, and Landau \cite{BCZ} introduced some of the tools used here
to resolve an old problem in frame theory:  {\em What is the correct
quantitative measure for redundancy for infinite dimensional
Hilbert spaces?}
In \cite{BCZ}, the following complementary result to Theorem
\ref{theorem:GaborExample1} is obtained.

\begin{theorem}\label{theorem:densityframe}
  Let $\varphi \in S_0(\R)$ and let $(\varphi,\Lambda)$ be a Gabor frame.
Then exists a set $\Lambda_\epsilon \subseteq \Lambda$
so that $(\varphi,\Lambda_\epsilon)$  is still a frame,
while $$\displaystyle D^+(\Lambda_\epsilon)\leq 1+\epsilon.$$
\end{theorem}

To prove results as Theorem~\ref{theorem:densityframe}  one has to  maintain
completeness  while removing large subsets from frames.  The challenge
when proving Theorem~\ref{theorem:MainForRiesz} is to obtain a given
lower Riesz bound while choosing as many elements as possible from a
Bessel system.

\section{Relative density and restricted invertibility}
\label{section:localization}

Definitions~\ref{definition:density} and \ref{definition:localization} are
based on the work
of Balan, Casazza, Heil, Landau \cite{BCHZ,BCHZ1,BCHZ2,BCZ}
(see also Gr\"ochenig \cite{G1}). They lead to a density concept
of subsets of $\F$ when $(\F,a,\G)$ is $\ell_1$-localized. The definition of
density of  $\F'\subseteq \F=\{f_i\}_{i\in I}$ relies on the localization
map $a:I\longrightarrow G$, as does the left hand side of (1) in Theorem~\ref{theorem:MainForRiesz}, while the right hand side of (1) in Theorem~\ref{theorem:MainForRiesz} does not depend on $a$.
In fact, as mentioned briefly in Section~\ref{Section:Notation}, in combination with
localized function systems though, $D^-(a;J)$ becomes independent of $a$.
This fact is well illustrated in the following example.
\begin{example}\rm
  Let $\E=\G=\{g_k\}_{k\in\Z}$ be an orthonormal basis of $\H$.
Let $T: \H\longrightarrow \H$ be defined by $Tg_k=g_{[\frac k 2]}$.
Let $\F=T\G$ and $a:\Z\mapsto \Z$ be so that
$$r(k)\geq |\langle f_{k''}, g_{a(k'')-k}\rangle|=|\langle g_{[\frac {k''} 2]},
g_{a(k'')-k}\rangle|=\delta([\tfrac {k''} 2]-a(k'')+k), \quad  k\in\Z.$$
Clearly, $r\in \ell_2(\Z)$ then implies $ [\frac {k''} 2]-a(k'')$, $k''\in \Z$,
is bounded.
Given $a_1$, $a_2$ with  $ [\frac {k''} 2]-a_1(k'')$, $k''\in \Z$, and
$ [\frac {k''} 2]-a_2(k'')$, $k''\in \Z$, bounded, then $a_1(k'')-a_2(k'')$,
$k''\in \Z$, bounded, and, clearly  $D^-(a_1;J)=D^-(a_2;J)$ for all
subsets $J\subseteq \Z$. (See Proposition~\ref{proposition:bounded} for a detailed argument.).
\end{example}

In general, for a family of functions $\F$ and a reference system  $\G$, each element  $f\in \F$ is naturally placed within $\G$ as the coefficient sequence $\{\langle f, g_k \rangle\}_k$  decays away from its {\em center of mass} as $\|k\|_\infty \rightarrow \infty$ by virtue of  $\{\langle f, g_k \rangle\}\in \ell_2(G)$. The function family $\F$ being $\ell_2$--localized with  respect to $\G$ simply means that  the decay behavior of $\{\langle f, g_k \rangle\}_k$  away from its {\em center of mass} is independent of $f\in \F$.

As each $f\in\F$ is local within the {\em coordinate system}  $\G$, an explicit location map $a:I\longrightarrow G$ is not needed. Localization and density of $\F$ with respect to $\G$ are fully determined by $\G$. To address this, we give a definition of localization and density which is independent of an explicit index set map $a:I\longrightarrow G$.

\begin{definition}\label{definition:ourlocalization} Let $p=1$ or $p=2$.
The  set  $\F\subseteq \H$ is $\ell_p$-{\bf localized} with respect to
$\G= \{g_k\}_{k\in  G}$ if there exists a  sequence
$r\in \ell_p(G)$ so that for each $f\in \F$ there is a $k\in G$
with $\langle f,g_n \rangle\leq r(n-k)$ for all $n\in G$.

The operator $T:\H'\longrightarrow \H$ is $\ell_p$-{\bf localized} if there
exists an orthonormal basis $\E$ of $\H'$ and a frame $\G$
of $\H$ so that $T(\E)$ is $\ell_p$-localized with respect to $\G$.
\end{definition}

Note, that any diagonalizable operator, for example, a  compact normal
operator on a separable Hilbert space is $\ell_1$-localized.

\begin{definition}\label{definition:ourdensity}
 The {\bf lower density} and {\bf upper density} of $\F$ with respect to
$\G$ are given, respectively, by
 \begin{eqnarray}
     D^-(\F;\G)&=&\liminf_{R\to\infty}
\inf_{k\in G} \frac{\sum_{f\in \F}a_f \sum_{n\in B_R(k)}
|\langle f,g_{n-k} \rangle|^2}{ |B_R(0)|} , \label{equation:lowerdensity}\\
    D^+(\F;\G)&=&\limsup_{R\to\infty}
\sup_{k\in G} \frac{\sum_{f\in \F} a_f \sum_{n\in B_R(k)}
|\langle f,g_{n-k} \rangle|^2}{ |B_R(0)|}, \label{equation:lowerdensity}
\end{eqnarray}
   where $a_f=\big(\sum_{n\in G} |\langle f,g_n \rangle|^2\big)^{-1}$,
$f\in \F$. If $D^-(\F;\G)= D^+(\F;\G)$,
then  $\F$ has {\bf uniform density}  $D(\F;\G)=D^-(\F;\G)= D^+(\F;\G)$ with
respect to $\G$ .
\end{definition}

Note that if $\mathcal \G$ is  a tight frame with upper and lower
frame bound $A$, then $a_f=(A\|f\|^2)^{-1}$ for  $f\in\F$. The following four propositions describe the relationship between Definitions~\ref{definition:density} and  \ref{definition:localization} and Definitions~\ref{definition:ourlocalization} and  \ref{definition:ourdensity}

\begin{proposition}\label{proposition:bounded} Let $p=1$ or $p=2$.
  If $(\F,a,\G)$ is $\ell_p$-localized, $\|f_i\|\geq \uu>0$ for all $i\in I$ and $\G$ is a frame, then  $(\F,b,\G)$ is $\ell_p$-localized if and only if $a-b:I\longrightarrow G$ is bounded.
\end{proposition}
\begin{proof}
  If $a-b$ bounded, then clearly $(\F,b,\G)$ is $\ell_p$-localized.

  To see the converse, let us assume that $a-b$ is not bounded while $(\F,a,\G)$  and  $(\F,b,\G)$ are $\ell_p$-localized. Choose $r\in \ell_p(G)$ with $|\langle f_i,g_k\rangle|\leq r(a(i)-k),r(b(i)-k)$ for all $i\in I, k\in G$.  Observe that
  \begin{eqnarray*}
    \sum_{k\in G} \min\{ r(k) , r(k-n) \}^2\longrightarrow 0\quad \text{as}\quad \|n\|_\infty \to\infty.
  \end{eqnarray*}
  Let $A$ be the lower frame bound of $\G$. Choose $M$ so that $\sum_{k\in G} \min\{ r(k) , r(k-n) \}^2\leq \tfrac 1 2 A\uu^2$ for all $n$ with  $\|n\|_\infty \geq M$ and choose $i$ with $\|a(i)-b(i)\|_\infty \geq M$. Then
  \begin{eqnarray*}
         0<A\uu^2 &\leq&   A\|f_i\|^2\leq \sum_{k\in G} |\langle f_i, g_k\rangle|^2 \leq \sum_{k\in G} \min\{ r(a(i)-k) , r(b(i)-k) \}^2 \\
         &= & \sum_{k\in G} \min\{ r(k) , r(k-(a(i)-b(i)) \}^2 \leq \tfrac 1 2 A\uu^2,
  \end{eqnarray*}
  a contradiction.
\end{proof}

\begin{proposition}
  If  $a-b$ is bounded, then $D^-(J,a)=D^-(J,b)$ and $D^+(J,a)=D^+(J,b)$ for all $J\subseteq I$.
\end{proposition}
\begin{proof}
Let $\|a(i)-b(i)\|_\infty \leq M$ for all $i\in I$ and choose $J\subseteq I$.

Clearly,
\begin{eqnarray}\qquad
  D^-(b;J){=}\liminf_{R\to\infty} \inf_{k\in G} \frac{|b^{-1} (B_R(k) )
\cap J|}{ |B_R(0)|}{=}\liminf_{R\to\infty} \inf_{k\in G} \frac{|b^{-1} (B_{R+M}(k) )
\cap J|}{ |B_R(0)|}. \label{equation:AddingABitDoesNotMatter}
\end{eqnarray}
Choose $k_m\in G$, $R_m\in\R^+$ with
$$D^-(b;J)=\lim_{m\to\infty} \frac{|b^{-1} (B_{R_m+M}(k_m) )
\cap J|}{ |B_R(0)|}$$

Now observe that due to the boundedness of $a-b$, we have  $a(j)\in B_R(k)$ implies $b(j)\in B_{R+M}(k)$ and we conclude that
$$  |a^{-1}(B_{R_m}(k_m)\cap J|\leq |b^{-1}(B_{R_m+M}(k_m)\cap J| $$ and so
\begin{eqnarray*}
D^-(a;J)&=&\liminf_{R\to\infty} \inf_{k\in G} \frac{|a^{-1} (B_R(k) )
\cap J|}{ |B_R(0)|} \leq \liminf_{n\to\infty} \frac{|a^{-1} (B_{R_m}(k_m) )
\cap J|}{ |B_{R_m}(0)|}\\  &\leq& \lim_{n\to\infty} \frac{|b^{-1} (B_{R_m+M}(k_m) )
\cap J|}{ |B_{R_m}(0)|} =D^-(b;J).\end{eqnarray*}
The inequalities $D^-(a;J)\geq D^-(b;J)$, $D^+(a;J)\leq D^+(b;J)$ and $D^+(a;J)\geq D^+(b;J)$ follow similarly.
\end{proof}

\begin{proposition}\label{proposition:BesselGivesFinite} Let $\F$ be Bessel with $\|f\|\geq \uu>0$ and $\G$ be a frame. If  $(\F,a,\G)$ is $\ell_2$-localized, then $D^+(a;I)<\infty$.
\end{proposition}
\begin{proof}
Let $B_{\F}$ be a Bessel bound of $\F$ and $A_{\G}, B_{\G}$ be frame bounds of $\G$. Choose $r\in \ell_2(G)$ with $|\langle f_i, g_n \rangle|\leq r(a(i)-n)$ for $i\in I, n\in G$, and  $M$ with $\sum_{n\notin B_M(0)} r(n)^2 \leq \tfrac 1 2 A_{\G}\uu^2$.  Suppose $D^+(a;I)=\infty$. Then exists for each $m\in\N$ an element $k_m\in G$ with $|a^{-1}(k_m)|\geq m$. We compute
\begin{eqnarray*}
 B_{\F} B_{\G} |B_M(0)| & \geq &  B_{\F} \sum_{n\in B_M(k_m)} \|g_n\|^2
 \geq  \sum_{i\in I} \sum_{n\in B_M(k_m)} |\langle f_i, g_n\rangle|^2  \\
 &\geq & \sum_{i\in a^{-1}(k_m) }\Big(\sum_{n\in G} |\langle f_i, g_n\rangle|^2 - \sum_{n\notin B_M(k_m)} |\langle f_i, g_n\rangle|^2 \Big)\\
&\geq&  \sum_{i\in a^{-1}(k_m) }\Big( A_{\G} \| f_i\|^2 - \sum_{n\notin B_M(k_m)} r(k_m-n)^2 \Big)\\
&\geq& m\Big( A_{\G} \uu^2 - \tfrac 1 2 A_{\G} \uu^2 \Big) \geq\tfrac 1 2 m  A_{\G} \uu^2.
\end{eqnarray*}
As the left hand side above is finite and independent of $m$ while the right hand side grows linearly with $m$, we have reached a contradiction.
\end{proof}

\begin{proposition}\label{proposition:equalityDensityConcepts}
  Let $\G$ be a frame and $(\F,a,\G)$ be $\ell_2$-localized where $\|f_i\|\geq\uu>0$, $i\in I$. Then for any $J\subseteq I$, $\F_J=\{f_j\}_{j\in J}$, we have
  $D^-(a;J)\leq      D^-(\F_J;\G)$ and $D^+(a;J)=
     D^+(\F_J;\G)$. If, moreover, $\F$ is Bessel, then $D^-(a;J)=
     D^-(\F_J;\G)$.
\end{proposition}
\begin{proof} Let $r\in \ell_2(G)$ be given with $|\langle f_i, g_n \rangle| \leq r(a(i)-n))$ for all $i\in I$, $n\in G$.
Let $A$ be the lower frame bound of $\G$. Then for all $i\in I$,
$$0<\uu^2 A\leq \sum_{n\in G} |\langle f_i,g_n \rangle|^2\leq \sum_{n\in G} r(a(i)-n)^2 = \|r\|^2,$$ so
 $\|r\|^{-2}\leq a_{f_i}\leq \uu^{-2} A^{-1}$.
  For $\epsilon>0$ choose $M$ so that
  $$ \uu^{-2} A^{-1}\sum_{n\notin B_M(0)} r(n)^2<\epsilon.$$
  For all $i\in I$ this implies
  $$ 1 - a_{f_i} \sum_{n\in B_M(a(i))} |\langle f_i,g_n \rangle|^2 =
  a_{f_i} \sum_{n\notin B_M(a(i))} |\langle f_i,g_n \rangle|^2 < \epsilon
  .$$

 Let $J\subseteq I$. For any $k$ and  $R>M$ we have
 \begin{eqnarray*}
(1-\epsilon)  |a^{-1} (B_{R-M} (k)) \cap J|
        &=&  \sum_{j\in J, a(j)\in B_{R-M}(k)} (1-\epsilon) \\ &\leq& \sum_{j\in J, a(j)\in B_{R-M}(k)} a_{f_j} \sum_{n\in B_R(k)}  |\langle f_j,g_n \rangle|^2 \\
        & \leq &  \sum_{j\in J} a_{f_j} \sum_{n\in B_R(k)}  |\langle f_j,g_n \rangle|^2\,.
 \end{eqnarray*}
Equation \eqref{equation:AddingABitDoesNotMatter} implies then   $D^-(a;J)\leq
     D^-(\F_J;\G)$ and $D^+(a;J)\leq
     D^+(\F_J;\G)$.

Note that if $D^+(a;J)=\infty$, then $D^+(a;J)=
     D^+(\F_J;\G)$ follows from  $D^+(a;J)\leq
     D^+(\F_J;\G)$.

To obtain $D^+(a;J)\geq
     D^+(\F_J;\G)$ if $D^+(a;I)<\infty$ and  $D^-(a;J)\geq
     D^-(\F_J;\G)$ if $\F$ is Bessel, we may assume $D^+(a;J)<\infty$ (see Proposition~\ref{proposition:BesselGivesFinite}). Then there exists $K\in\N$ with $|a^{-1}(k)|\leq K$ for all $k\in G$. For $\epsilon >0$ and $M$ sufficiently large, we have for $n\in B_R(k)$, $k\in G$,
\begin{eqnarray*}
  \sum_{j\in J, a(j)\notin B_{R+M}(k)} |\langle f_j, g_n  \rangle|^2 & \leq& \sum_{j\in J, a(j)\notin B_{R+M}(k)} r(a(j)-n)^2 \\ &\leq& K\sum_{m\notin B_{R+M}(k)} r(m-n)^2 \leq \epsilon.
\end{eqnarray*}
We conclude for $k\in G$ and $R$ large that
     \begin{eqnarray*}
     \sum_{j\in J} a_{f_i} \sum_{n\in B_R(k)}  |\langle f_j,g_n \rangle|^2
        &\leq &  \sum_{j\in J, a(j)\in B_{R+M}} a_{f_i} \sum_{n\in B_R(k)}  |\langle f_j,g_n \rangle|^2 \\
        && \qquad +\sum_{j\in J, a(j)\notin B_{R+M}} a_{f_i} \sum_{n\in B_R(k)}  |\langle f_j,g_n \rangle|^2 \\
        &\leq& |a^{-1} (B_{R+M} (k)) \cap J| \\
        &&\qquad + \sum_{n\in B_R(k)} \sum_{j\in J, a(j)\notin B_{R+M}} a_{f_i}  |\langle f_j,g_n \rangle|^2\\
        &\leq& |a^{-1} (B_{R+M} (k)) \cap J| + \uu^{-2}A^{-1} B_R(0) \epsilon,
 \end{eqnarray*}
 and $D^-(a;J)\geq
     D^-(\F_J;\G)$,  $D^+(a;J)\geq
     D^+(\F_J;\G)$ follows.
\end{proof}

The following example illustrates the role of the Bessel bound of $\F$ to achieve $D^-(a;J)=
     D^-(\F_J;\G)$.
\begin{example} \rm
  Let $\G=\{e_k\}_{k\in\Z}$ be an orthonormal basis, and let the members of $\F$ be given by $f_i=f=\sum_{m\in \Z}^\infty 2^{-|m|} e_m$, $i\in \Z$. For $a:\Z\longrightarrow \Z$, $i \mapsto 0$, we have $(\F,a,\G)$ is $\ell_1$-localized, $D^-(a;\Z)=0$, but
  \begin{eqnarray*}
  D^-(\F;\G)&=&
  \liminf_{R\to\infty} \inf_{k\in \Z} \frac{\sum_{f\in \F} \sum_{n\in B_R(k)}
|\langle f,e_{n-k} \rangle|^2}{ |B_R(0)|} \\ &=& \liminf_{R\to\infty}
\inf_{k\in \Z} \frac{\sum_{i \in \Z} \sum_{n\in B_R(k)}
2^{- |n-k| } }{ |B_R(0)|}\\ &=& \liminf_{R\to\infty}
\inf_{k\in \Z}\infty =\infty.
\end{eqnarray*}
\end{example}

\begin{definition}\label{definition:RelativeDensity}
Let $\F$ be $\ell_1$-localized with respect to $ \G$
with $0<D^-(\F;\G)\leq D^+(\F;\G)< \infty$. The  {\bf relative lower density}, respectively {\bf relative upper density} of
$\F'\subseteq \F$ is
\begin{eqnarray}
   R^-(\F',\F;\G)&=&\frac{D^-(\F';\G)}{ D^{+}(\F;\G)} , \label{equation:UpperRelativeDensity} \\  R^+(\F',\F;\G)&=&\frac{D^+(\F';\G)}{ D^{-}(\F;\G)}
\label{equation:LowerRelativeDensity}.
\end{eqnarray}
If $R^-(\F',\F;\G)= R^+(\F',\F;\G)$, then we say that $\F'$ has {\bf uniform relative density} $R(\F',\F;\G)= R^+(\F',\F;\G)$ in  $\F$.
\end{definition}
%
%

 Examples~\ref{example:localization1} and \ref{example:localization2}
below  illustrate  the interaction  of density and localization.
We are now ready to restate  the main result of the paper.

\begin{theorem}\label{theorem:MainResult} Let $\cc$ be the
function provided in Theorem~\ref{theorem:RestrictedInvertTheoremFinite}.
Let $\F\subseteq \H$ be $\ell_1$-localized with respect to the frame $\G$
and assume that  $\|f\|\geq \uu$ for all $f\in \F$ and $\F$ is Bessel
with Bessel bound $B_\F$. Assume either \\[.1cm]

\noindent (A) $\G = \{g_k: k\in  G\}$ is a Riesz basis for $\H$
with Riesz bounds $A_\G,B_\G$, \\[.1cm]

or\\[.1cm]

\noindent (B) $\G = \{g_k: k\in G\}$ is a frame for $\H$ with
$\ell_1$- self-localized dual frame
$\widetilde \G = \{\widetilde g_k: k\in G\}$.\\[.1cm]

\noindent If $(\F;\G)$ is $\ell_1$-localized with
$
    0<D^-(\F;\G)\leq D^+(\F;\G)<\infty.
$
 Then for every $\epsilon>0$ and $\delta>0$ there is a subset
$ \F_{\epsilon\delta} \subseteq \F$ of uniform density with \\[.1cm]

\noindent (1)  \quad $\displaystyle
  R^+(\F_{\epsilon\delta},\F;\G)=\frac{D(\F_{\epsilon\delta};\G)}
{ D^{-}(\F;\G)}  \ge \frac{(1-\epsilon)\uu^2}{B_\F}$, \\[.1cm]


\noindent (2) \quad  $\F_{\epsilon\delta} $ is a Riesz  sequence
with Riesz bounds \\
\indent \quad (A)\quad $\displaystyle  \cc(\epsilon) (1-\delta)
\frac {A_\G} {B_\G}\, \uu^2$, $B_\F$. \\
 \indent \quad (B)\quad $\displaystyle \cc(\epsilon) (1-\delta)\,
\uu^2$, $B_\F$.
%
\end{theorem}

\begin{proof} Note that for $\F=\{f_i\}_{i\in I}$, $J_{\epsilon\delta}\subseteq I$, and $\F_{\epsilon\delta}=\{f_j\}_{j\in J_{\epsilon\delta}}$, we have in general
 $$  \frac{D(a;J_{\epsilon\delta})}{ D^{-}(a;I)} \geq \frac{D(\F_{J_{\epsilon\delta}};\G)}
{ D^{-}(\F;\G)} =R^+(\F_{\epsilon\delta},\F;\G) . $$
But under the given assumptions, Proposition~\ref{proposition:equalityDensityConcepts} implies equality above, and, hence,  Theorem~\ref{theorem:MainResult} is  a restatement of Theorem~\ref{theorem:MainForRiesz}.
\end{proof}

Theorem~\ref{theorem:MainResult}
can be rephrased in terms of $\ell_1$-localized operators.  Again, given
an $\ell_1$-localized operator $T:\H'\longrightarrow \H$ and respective
orthonormal basis $\E$ of $\H'$ and a frame $\G$ of $\H$ with
$\F=T(\E)$ being $\ell_1$-localized with respect to $\G$, then
Theorem~\ref{theorem:MainResult} holds verbatim with the Bessel bound
$B_\F$ being replaced with $\|T\|^2$.

\begin{theorem}[{\bf General Infinite Dimensional Restricted Invertibility
Theorem}]
\label{theorem:RestrictedInvertTheoremInfinite2}

Let $\cc$ be the function provided in
Theorem~\ref{theorem:RestrictedInvertTheoremFinite}.
Let $\E$ and $\G =\{g_k\}_{k\in G}$ be orthonormal
bases for an Hilbert space $\H$,
let $T:\H \rightarrow \H$ be a bounded linear operator satisfying
$\|Te\| = 1$, for all $e\in \E$. Assume that  $T\E$ is $\ell_1$-localized with
respect to $\{g_k\}_{k\in G}$.  Then for all $\epsilon,\delta >0$,
there is a subfamily $\E_{\epsilon\delta}\subseteq \E$ of uniform density
with \\[.1cm]

\noindent (1) $\displaystyle  R^+(T\E_{\epsilon\delta},T\E;\G) \ge
\frac{1-\epsilon}{\|T\|^2}$, and \\[.1cm]

\noindent (2)  $T\E_{\epsilon\delta}$ is  a Riesz system with Riesz bounds
$ \cc(\epsilon) (1-\delta)$, $\|T\|^2$.
\end{theorem}

%

We close this section with two examples displaying the interaction of density, localization, and Theorems~\ref{theorem:MainResult} respectively \ref{theorem:RestrictedInvertTheoremInfinite2}.

\begin{example}\label{example:localization1}\rm
In the following, we shall consider as reference system for $\H$
\begin{itemize}
  \item an orthonormal basis $\G=\{g_n\}_{n\in\Z}$ of  $\H$;
  \item the system $\G' =\{g'_n\}$ given by $g'_{2n}=g'_{2n+1}=g_n$,
$n\in\Z$, that is, $\G'$ consists of two intertwined copies of $\G$;
  \item the system $\G'' $  given by the sequence
$$ \ldots ,e_{-7}, e_{-2},e_{-5},e_{-3},e_{-1}, e_0,e_1,e_3,e_5,e_2,e_7,e_9,
e_{11}, e_4, e_{13},e_{15},e_{17},e_6,e_{19},\ldots   $$
\end{itemize}.

Moreover, we shall  consider the operators $T_1,T_2,T_3:\H\longrightarrow \H$
given by
\begin{itemize}
  \item $T_1e_n=e_n$, $n\in \Z$, that is,  $T_1={\rm Id}$;
  \item $T_2 e_n= e_{2[\frac n 2]}$, $n\in \Z$;
  \item $T_3 e_0=e_1$, $T_3 e_n =e_n$ for $n\in \Z\setminus \{0\}$;
\end{itemize}
Clearly, $\|T_1\|=1$ and $\|T_2\|=\|T_3\|=\sqrt{2}$.
Note that the right hand side in Theorem~\ref{theorem:RestrictedInvertTheoremInfinite2} (1) is
$(1-\epsilon)$ for $T_1$ and $(1-\epsilon)/2$ for $T_2,T_3$. Set
$\F=\{e_n\}_{n\in\Z}$, $\F_{\rm even}=\{e_{2n}\}_{n\in\Z}$, and
observe that they form orthonormal bases for their closed linear span. Hence, we could choose $\F_{\epsilon\delta}=\F\subseteq \F$
in case of $T_1$, and  $\F_{\epsilon\delta}=\F_{\rm even}$ in case
of $T_1,T_2$. We now discuss strengths and shortcomings of Theorem~\ref{theorem:RestrictedInvertTheoremInfinite2} when using as reference systems $\G$, $\G'$, and $\G''$.

  \begin{enumerate}

\item Clearly, $\F$ is $\ell_1$--localized with
respect to $\G$.
Moreover $D^-(\F;\G)=D(\F;\G)=1$, $D(\F_{\rm even};\G)=\frac 1 2$,
and $$R(\F,\F;\G)= D(\F;\G) / D(\F;\G)=1,$$
$$R(\F_{\epsilon\delta},\F;\G)= D(\F_{\rm even};\G) / D(\F;\G)=
\frac 1 2.$$  So $\F$, $\F_{\rm even}$ satisfy the conclusions of  (1) in Theorem~\ref{theorem:RestrictedInvertTheoremInfinite2}  for $T_1$ respectively $T_2$ and $T_3$.

\item $\F$ is $\ell_1$--localized with
respect to $\G' $.
We have $D^-(\F;\G' )=D(\F;\G' )=\frac 1 2$ and
$D(\F_{\rm even};\G' )=\frac 1 4$, so again
$R(\F,\F;\G' )= D(\F;\G') / D(\F;\G')=1$ and
$R(\F_{\epsilon\delta},\F;\G' )= D(\F_{\rm even};\G' ) /
D(\F;\G' )=\frac 1 2$ for $T_2,T_3$, satisfying
Theorem~\ref{theorem:RestrictedInvertTheoremInfinite2} (1) for $T_1$ respectively $T_2$ and $T_3$.

\item
$\F$ is also $\ell_1$--localized with respect to $\G'' $. Now,
$D^-(\F;\G'' )=D(\F;\G'' )=1$, but $D(\F_{\rm even};\G'' )=\frac 1 4$,
consequently, $R(\F,\F;\G'' )= D(\F;\G'' ) / D(\F;\G'' )=1$, so
Theorem~\ref{theorem:RestrictedInvertTheoremInfinite2} (1) for $T_1$ is satisfied, but as
$$R(\F_{\epsilon\delta},\F;\G'' )= D(\F_{\rm even};\G'' ) / D(\F;\G'' )=
\frac 1 4,$$ so $\F_{\epsilon\delta}=\F_{\rm even}$ is not a valid choice
satisfying Theorem~\ref{theorem:RestrictedInvertTheoremInfinite2} (1) for $T_2,T_3$.
Theorem~\ref{theorem:RestrictedInvertTheoremInfinite2} guarantees for any $\epsilon,\delta >0$
the existence of a Riesz sequence $\F_{\epsilon\delta}$ with
$R(\F_{\epsilon\delta},\F;\G'' )= D(\F_{\epsilon\delta};\G'' ) / D(\F;\G'' )\geq
 (1-\epsilon) \frac 1 2$, and clearly, in the case of $T_2$ and $T_3$,
we may choose $\F_{\epsilon\delta}=\F_{\rm odd}=\F\setminus \F_{\rm even}$.
Then we  have the seemingly better result
$R(\F_{\epsilon\delta}=\F_{\rm odd},\F;\G'' )= D(\F_{\rm odd};\G'' ) /
D(\F;\G'' )=\frac 3 4 >(1-\epsilon) \frac 1 2$.

\item Note that regardless of how we adjust $\G$, we will not be
guaranteed a Riesz system as large as the optimal one for $T_3$, namely
$\F_{\epsilon\delta}=\F\setminus\{Te_0\}$. Clearly, this shortcoming is
shared by the finite dimensional version of Bourgain-Tzafriri.

  \end{enumerate}

The following example illustrates that the possible choices of index set $G$ of $\G$ is strongly influenced by $\F$ in Theorem~\ref{theorem:MainResult} respectively $T$ in Theorem~\ref{theorem:RestrictedInvertTheoremInfinite2}.

\end{example}

\begin{example}\label{example:localization2} \rm
Consider the operator $T_4:\H\longrightarrow \H$ given by
 $T_4 e_n=e_n+e_{2n}$, $n\in\Z$.
 We have $\|T_4e_n\|\geq \uu=\sqrt{2}$ for $n\in\Z$ and
$$\|T_4(\sum c_n e_n)\|= \|\sum c_n e_n + \sum_n c_n e_{2n}\|\leq
\|\sum c_n e_n\| + \|\sum_n c_n e_{2n}\| =2\|\sum c_n e_n\|.$$ As
$\|T_4 e_0\|=\|2 e_0\|=2$, we have $\|T_4\|=2$. Note that also $$\|T_4
( N^{-\frac 1 2 }\sum_{n=1}^N e_{2^n})\|=N^{-\frac 1 2 }
\| e_2+2e_4+ 2e_8+\ldots +2e_{2^{N-1}}+e_{2^N}\|=\sqrt{\tfrac {4(N-2)+2} N}
\rightarrow 2$$ as  $ N\to\infty$.

The right hand side in Theorem~\ref{theorem:MainResult} (1) is
$(1-\epsilon)/2$ for $T_4$, and the orthogonal family
$\F_{\epsilon\delta}=T_4\{e_{2n+1}\}_{n\in\Z}$ satisfies the conclusions of
Theorem~\ref{theorem:MainResult} (2).
But $T_4(\E)$ is not $\ell_1$-localized with respect to  $\G$ whenever $\G$ is a linear ordering of $\E=\{e_n\}_{n\in\Z}$.  To see this, presume that  $T_4\E$ is  $\ell_1$-localized with respect to  $\G=\{ g_n =e_{\sigma(n)} \}_{n\in \Z}$ where $\sigma$ is a permutation on $\Z$. Let $r\in \ell_1(\Z)$ be the respective bounding sequence and choose $N$ so that $r(k)<1$ for $|k|\geq N$.
   Now, for some $k_{2N}\in\Z$, we have
    $$ \delta_{2N,\sigma(n)}+\delta_{4N,\sigma(n)}= \langle Te_{2N}, g_n \rangle\leq r(k_{2N}-\sigma(n)),\quad n\in \Z.$$
    Inserting $n_1=\sigma^{-1}(2N)$ respectively $n_2=\sigma^{-1}(4N)$, we obtain $1\leq r(k_{2N}-2N)$ respectively $1\leq r(k_{2N}-4N)$, and, by choice of $N$, $|k_{2N}-2N|,|k_{2N}-4N|< N$, leading to the contradiction $2N<2N$.

    As an alternative to linear orders on $\E$, consider the following as reference system $\G_{\Z^2}$
    \begin{eqnarray}
      \begin{array}{llllllllllllll}
       & &  & & & & \vdots &  & & & & & \\
         & e_{-208} & e_{-104} & e_{-52} & e_{-26} & e_{-13} & e_{13} & e_{26} & e_{52} & e_{104} & e_{208} & e_{416} & \\
         & e_{-144} & e_{-72} & e_{-36} & e_{-18} & e_{-9} & e_{9} & e_{18} & e_{36} & e_{72} & e_{144} & e_{288} &  \\
        & e_{-80} & e_{-40} & e_{-20} & e_{-10} & e_{-5} & e_{5} & e_{10} & e_{20} & e_{40} & e_{80} & e_{160} &  \\
         \hdots &e_{-16} & e_{-8} & e_{-4} & e_{-2} & e_{-1} & e_{0} & e_{1} & e_{2} & e_{4} & e_{8} & e_{16} & \hdots \\
         & e_{-48} & e_{-24} & e_{-12} & e_{-6} & e_{-3} & e_{3} & e_{6} & e_{12} & e_{24} & e_{48} & e_{96} & \\
         & e_{-112} & e_{-56} & e_{-28} & e_{-14} & e_{-7} & e_{7} & e_{14} & e_{28} & e_{56} & e_{112} & e_{224} & \\
         & e_{-176} & e_{-88} & e_{-44} & e_{-22} & e_{-11} & e_{11} & e_{22} & e_{44} & e_{88} & e_{176} & e_{352} &  \\
         & e_{-240} & e_{-120} & e_{-60} & e_{-30} & e_{-15} & e_{15} & e_{30} & e_{60} & e_{120} & e_{240} & e_{480} &\\
          & &  & & & & \vdots &  & & & & &
      \end{array}.\notag
    \end{eqnarray}

    Clearly, $T_4$ is $\ell_1$-localized with respect to $\G_{\Z^2}$. In fact, we can choose $r=\delta_{(0,0)}+\delta_{(0,1)}\in \ell_1(\Z^2)$.

    Theorem~\ref{theorem:MainResult} guarantees for $\delta,\epsilon >0$  the existence of a Riesz sequence $\F_{\epsilon\delta}$ with $R(\F_{\epsilon\delta},\F;\G_{\Z^2})= D(\F_{\epsilon\delta};\G_{\Z^2}) / D(\F_{\epsilon\delta};\G_{\Z^2}) \geq (1-\epsilon)/2$.  We have $D(\F;\G_{\Z^2})=1$, but for the natural choice $\F_{\epsilon\delta}=T_4\{e_{2n+1}\}_{n\in\Z}$, we have  $D(\F_{\epsilon\delta};\G_{\Z^2}) =0$.  For $\F_{\epsilon\delta}=T_4 \{ e_{2^{2k}(2n+1)}\}_{n\in\Z, k\in \N_0}$, we have $D(\F_{\epsilon\delta};\G_{\Z^2}) =\frac 1 2$, therefore satisfying the conclusions of Theorem~\ref{theorem:MainResult}.

    For completeness sake, note that $T_4(\E)$ itself is not a Riesz
sequence.  To see this, observe that
    $$ \|\sum_{n=1}^N (-1)^n T_4e_{2^n}\|^2=\|\sum_{n=1}^N (-1)^n (e_{2^n}+e_{2^{n+1}})\|^2=\|-e_1+(-1)^N e_{2^{N+1}}\|^2=2
    $$
    while $\sum_{n=1}^N |(-1)^n|^2 =N$.
\end{example}

\section{Proof of Theorem~\ref{theorem:MainForRiesz}}

Note that  the generality assumed here, namely
that $G$ is any finitely generated Abelian group, is quite useful
in practice as the group is often given by the structure of the problem at hand.
For example, in time--frequency analysis, the group $G=\Z^{2d}$ is generally
used when considering single window Gabor systems. If we consider
multi-window Gabor systems, then an index set $\Z^{2d}\times H$ with
$H$ being a finite group is natural. (Also, see Example~\ref{example:localization2} for the dependence of $G$ on $T$ and $\F$.)
The following proposition will allow us to consider in our proofs only localization with respect to $\G$ with  $G= \Z^d$.

\begin{proposition}\label{P1}
Let $H$ be a finite Abelian group of order $N$ and $G=\Z^d\times H$. Choose a bijection $u:\{ 0,1,\ldots,N{-}1 \}\longrightarrow H$ and $$U:\Z^d\longrightarrow G,\quad (k_1,\ldots, k_d)\mapsto (k_1,\ldots, k_{d-1}, \lfloor k_d/N \rfloor, u(k_d\mod N)).$$ For  $a:I \longrightarrow G$ set $b=U^{-1}\circ a:I\longrightarrow \Z^d$. Then
\begin{enumerate}
  \item $D^{-}(b;J)= D^{-}(a;J)$ and $D^{+}(b;J)= D^{+}(a;J)$ for all $J\subseteq I$.
  \item $(\F,a,\G)$ is  $\ell_1$-localized  if
and only if $(\F,b,\G')$ is  $\ell_1$-localized where
$\G'=\{g_{U(k)}\}_{k\in\Z^d}$.
\end{enumerate}
\end{proposition}

\begin{proof} First, observe that for all $P\in\N$ we have
$$     D^-(a;J)=\liminf_{R\to\infty} \inf_{k\in G} \frac{|a^{-1} (B_R(k) )
\cap J|}{ |B_R(0)|}=\liminf_{M\to\infty} \inf_{k\in G} \frac{|a^{-1} (B_{MP}(k) )
\cap J|}{ |B_{MP}(0)|}$$ where $M\to\infty$, $M\in \N$.

Note that for $R=MN$, $M\in\N$,
\begin{eqnarray*}
  \big| b^{-1}\big(B_R^{Z^d}(k)\big)\cap I\big|
        &=&\big| a^{-1}\circ U \big(B_{MN}^{Z^d}(k)\big)\cap I\big| \\
        &=& \big| a^{-1} \big(B_{MN}^{Z^{d-1}}(k_1,\ldots,k_{d-1})\times B^\Z_{M}({\lfloor \tfrac {k_d} N\rfloor})\times H  \big)\cap I\big|.
\end{eqnarray*}
Now, compare
\begin{eqnarray*}D^-(b;J)&=&\liminf_{M\to\infty} \inf_{k\in \Z^d} \frac{ \big| b^{-1}\big(B_{MN}^{\Z^d}(k)\big)\cap I\big| }{ |B_{MN}^{\Z^d}(0)|}\\
&=&\liminf_{M\to\infty} \inf_{k\in \Z^d} \frac{ \big|a^{-1} \big(B_{MN}^{\Z^{d-1}}(k_1,\ldots,k_{d-1})\times B^\Z_{M}({\lfloor \tfrac {k_d} N\rfloor}) \times H \big)\cap I \big| }{(2MN+1)^d}
\end{eqnarray*}
and
\begin{eqnarray*}D^-(a;J)&=&\liminf_{M\to\infty} \inf_{(k,h)\in \Z^d\times H} \frac{ \big| a^{-1}\big(B_{MN}^{Z^d\times H}(n,h)\big)\cap I\big| }{ |B_{MN}^{\Z^d\times H}(0)|}\\
&=&\liminf_{M\to\infty} \inf_{k\in \Z^d} \frac{ \big| a^{-1} \big(B_{MN}^{\Z^{d-1}}(k_1,\ldots,k_{d-1})\times B^\Z_{MN}(k_d) \times H \big)\cap I\big| }{N(2MN+1)^d}.
\end{eqnarray*}
  As the sets $B_R(k)=RB_1(0) + k$ in Definition~\ref{definition:density} can be replaced by sets of the form $D_R(k)=RD + k$ if $D$ is a compact set of measure 1 and 0 measure boundary (Lemma 4 in \cite{Lan67}), we conclude that $D^-(b;J)=D^-(a;J)$, and, similarly $D^+(b;J)=D^+(a;J)$.

The second assertion is obvious. \end{proof}

%
%

\begin{lemma}\label{lemma:ApplySchurLocalization}
  Let $(\F=\{f_i\}_{i\in I}, a, \G=\{g_k\}_{k\in G})$ be $\ell_1$-localized
with $D^+(a;I)<\infty$.  For $M^R:\ell_2(G)\mapsto \ell_2(I)$ given by
$(M^R)_{i,k}=\langle f_i, g_k \rangle$ if $\|a(i)-k\|_\infty> R$, and
$(M^R)_{i,k}=0$ otherwise, we have
 \[
    \lim_{R\rightarrow \infty}\|M^R\| = 0.\]
\end{lemma}
\begin{proof}
  As $(\F=\{f_i\}_{i\in I}, a, \G=\{g_k\}_{k\in G})$ is $\ell_1$-localized,
there exists $r\in \ell_1(G)$ with
  \begin{eqnarray}
      \notag r(k)\geq |\langle f_i, g_{k'}\rangle| \quad \text{ if}\quad a(i)-k'=k.
  \end{eqnarray}
Hence,
\begin{eqnarray}
  \sup_{i\in I} \sum_{k\in G} |(M^R)_{i,k}|&=& \sup_{i\in I}
\sum_{k': \|a(i)-k'\|_\infty>  R} |\langle f_i, g_{k'}\rangle|
  \leq \sup_{i\in I} \sum_{k': \|a(i)-k'\|_\infty>  R} r(a(i)-k') \notag \\
   &\leq& \sup_{i\in I} \sum_{k: \|k\|_\infty> R} r(k) =:\Delta_r(R). \notag
\end{eqnarray}
Similarly, setting $K=\max_{k\in G} |a^{-1}(k)|$ (it is finite since
$D^+(a;I)<\infty$) we obtain
\begin{eqnarray}
  \sup_ {k\in G} \sum_{i\in I} |(M^R)_{i,k}|&=& \sup_{k'\in G}
\sum_{i: \|a(i)-k'\|_\infty>  R} |\langle f_i, g_{k'}\rangle|
  \leq \sup_{k'\in G} \sum_{i: \|a(i)-k'\|_\infty>  R} r(a(i)-k') \notag \\
   &\leq& \sup_{k'\in G} K \sum_{k: \|k\|_\infty> R} r(k) =K\Delta_r(R). \notag
\end{eqnarray}
The result
now follows from Schur's criterion \cite{K,R}
since $\Delta_r(R)\longrightarrow 0$.
\end{proof}

The following  lemma  is similar to Lemma 3.6
of \cite{BCZ}.

\begin{lemma}\label{lemma:AnalysisOperatorsConverge}
Let $\G = \{g_k\}_{k\in  \Z^d}$ be a frame for $\H$
with dual frame $\widetilde \G =
\{\widetilde g_k\}_{k\in  \Z^d}$ and let $a:I \rightarrow G$
be a localization map of finite
upper density so that the Bessel system $(\{f_i\}_{i\in I},a,\G)$
is $\ell_1$-localized.  For $R>0$ set $$\displaystyle f_{iR}=
\sum_{n\in \mathbb Z^d:\, \|a(i)-n\|_\infty<R} \ \langle f_i , g_n
\rangle \widetilde g_n$$ and set
  \begin{eqnarray*}
    L_I: \H\longrightarrow \ell_2(I),\ h \mapsto \{\langle h,
f_i\rangle \} \text{ and } L_{IR}: \H\longrightarrow \ell_2(I),\ h
\mapsto \{\langle h, f_{iR}\rangle \}.
  \end{eqnarray*}
  Then
\[ \lim_{R\rightarrow \infty}\| L_I-L_{IR}\|  = 0.\]
\end{lemma}
\begin{proof}
  For $h\in \H$, we compute
  \begin{eqnarray}
    \|(L_I-L_{IR})h \|^2_{\ell_2}&=&\sum_{i\in I} |\langle h, f_{i}
\rangle- \langle h, f_{iR}\rangle|^2\
= \sum_{i\in I} |\langle h, f_{i}- f_{iR}\rangle|^2\notag\\
        &=& \sum_{i\in I} |\langle h, \hspace{-.3cm} \sum_{\|a(i)-n\|_\infty> R }
\langle f_{i},g_n\rangle \widetilde g_n\rangle|^2 = \sum_{i\in I} |\sum_{\|a(i)-n\|_\infty> R }\langle f_{i},g_n\rangle
\langle h, \widetilde g_n\rangle|^2\notag\\
        &=& \|M^R \{\langle h, \widetilde g_n\rangle\}_{n\in\mathbb Z^d}\|^2,\notag
  \end{eqnarray}
 with $M^R$ given by $M_{i,n}=\langle f_{i},g_n\rangle$ if $\|a(i)-n\|_\infty> R$
and $M_{i,n}=0$ otherwise.
 Since  $(\{f_i\}_{i\in I},a,\G)$ is $\ell_1$-localized, we can apply
Lemma~\ref{lemma:ApplySchurLocalization} and obtain
$\|M^R\| \longrightarrow 0$ as $R\rightarrow \infty$.
The result now follows from the boundedness of the map
$h\mapsto \{\langle h, \widetilde g_n\rangle\}_n$.
\end{proof}

\subsection{Proof of Theorem ~\ref{theorem:MainForRiesz}, assuming (A)}

  \vspace{.5cm}
  Fix $\epsilon,\delta > 0$. As $\cc$ given in Theorem~\ref{theorem:RestrictedInvertTheoremFinite} is positive and continuous, we can choose   $\epsilon'>0$ with $\epsilon'<\epsilon$ and
  \begin{eqnarray}
        \cc(\epsilon)(1-\delta)\leq \cc(\epsilon')(1-\tfrac \delta 2). \notag
  \end{eqnarray}
  Choose $\alpha>0$ satisfying $\alpha\leq \frac \delta {8}$ and
  \begin{eqnarray}
    (1-\epsilon')\frac{(1-\alpha)^2}{(1+\alpha)^2}\geq (1-\epsilon).\notag
  \end{eqnarray}
  Recall that for $R\in\mathbb N$,
  \begin{eqnarray*}
    D^-(a;I)=\liminf_{R\to\infty} \inf_{k\in \mathbb Z^d}
\frac{\big| a^{-1}(B_R(k))\cap I\big| }{\big|B_R(0)\big|}
    =\liminf_{R\to\infty} \inf_{k\in \mathbb Z^d}
\frac{\big| a^{-1}(B_R(k))\cap I\big| }{(2R+1)^d}
  \end{eqnarray*}
  where $B_R(k)=\{ k': \|k-k'\|_0 \leq R \}$.
Hence, we may choose $P>0$ such that for all $R\geq P$, $k\in\mathbb Z^d$,
we have
    \begin{eqnarray}\big| a^{-1}(B_R(k))\cap I\big| \geq (1-\alpha)\,
D^-(a;I)\,(2R+1)^d\, . \notag
  \end{eqnarray}

Let $\{\widetilde g_n\}_{n\in G}$ be the dual basis of $\{ g_n\}_{n\in G}$.
For any $R>0$,  set
\[ f_{iR}=\sum_{n\in \mathbb Z^d:\,
\|a(i)-n\|_\infty<R} \ \langle f_i , g_n \rangle \widetilde g_n.\]
 For
  \begin{eqnarray*}
    L_I: \H\longrightarrow \ell_2(I),\ h \mapsto \{\langle h,
f_i\rangle \} \text{ and } L_{IR}: \H\longrightarrow \ell_2(I),\ h
\mapsto \{\langle h, f_{iR}\rangle \},
  \end{eqnarray*}
  Lemma~\ref{lemma:AnalysisOperatorsConverge} implies that there
is $Q>0$ with the property that for all $R\geq Q$ we have
  \begin{eqnarray}
    \|L_I- L_{IR}\| \leq \min\Big \{\alpha\uu,\alpha \|T\|,
\Big(\frac {A\delta\,\cc(\epsilon')\uu^2} {8B}\Big)^{\frac 1 2}\Big\}.
\label{equation:operatornomLR}
  \end{eqnarray}

  Also, since
  \begin{eqnarray}
    D^+(a;I)=\limsup_{R\to\infty} \sup_{k\in \mathbb Z^d}
\frac{\big| a^{-1}(B_R(k))\cap I \big| }{\big|B_R(k)\big|}
    <\infty,\notag
  \end{eqnarray}
  we can pick $K>0$ with $|a^{-1}(k)|<K$, for all $k\in \mathbb Z^d$.

  By possibly increasing $P$ and $Q$, we can assume $P>Q$ and
  \begin{eqnarray}
    &K&\big( (2 P+1)^d-(2(P-Q)+1)^d\big)\notag \\
    &&\qquad \leq \alpha\,\uu^2 \,
\frac{(1-\epsilon')(1-\alpha)}{(1+\alpha)^2} \frac{D^-(a;I)(2P+1)^d }
{\|T\|^2}\label{equation:DropBorderRegions1}
  \end{eqnarray}

Set $\F_{k}=\{f_i:\ a(i)\in B_{P}(k) \}$ and correspondingly
$\F_{kQ}=\{f_{iQ}:\ a(i)\in B_{P}(k) \}$.
Equation (\ref{equation:operatornomLR}) implies
\begin{eqnarray*}
  \| f_i - f_{iQ} \|=\|L_I^\ast \{\delta_i\} - L_{IQ}^\ast
\{\delta_i\} \|=\|(L_I - L_{IQ})^\ast \{\delta_i\} \|\leq
\|L_I - L_{IQ}\|\leq \alpha\uu,
\end{eqnarray*}
and, therefore, $ \|f_{iQ}\|\geq(1-\alpha)\uu  $.
Similarly, we conclude for $T_Q: e_i\mapsto f_{iQ}$, that $\|T-T_Q\|<\alpha\|T\|$
and for $h=\sum a_i e_i \in \H$,
\begin{eqnarray}
  \| (T - T_Q) h \| &=& \|\sum a_i (T-T_Q)e_i \| =
\|\sum a_i (f_i - f_{iQ}) \|\notag \\ &=&
\|(L_I-L_{IQ})^\ast \{a_i\} \|  \leq \alpha\|T\|
\|\{a_i\}\|=\alpha \|T\| \|h\|. \notag
\end{eqnarray}

Applying Theorem~\ref{theorem:RestrictedInvertTheoremFinite}
to the finite sets $\F_{kQ}$ with cardinality $$n\geq
(1-\alpha)D^-(a;I)\, (2P+1)^d$$ and $\epsilon'$, we obtain
Riesz  sequences $\F'_{kQ}\subseteq \F_{kQ}$ with
\begin{eqnarray}
  |\F'_{kQ}| &\geq& (1-\epsilon')\uu^2\, (1-\alpha) \,
D^-(a;I)\, (2P+1)^d / \|T_Q\|^2 \notag \\
& \geq& \frac{(1-\epsilon')(1-\alpha)}{(1+\alpha)^2} \,
\uu^2\,D^-(a;I)\, (2P+1)^d / \|T\|^2
\end{eqnarray}
and lower Riesz bounds
$\cc(\epsilon')(1-\alpha)^2\uu^2.$

We further reduce $\F'_{kQ}\subseteq \F_{kQ}$ by setting
\begin{eqnarray}
  F''_{kQ} =  F'_{kQ}\cap a^{-1}(B_{P-Q}(k)). \label{equation:ensureRiesz}
\end{eqnarray}
Now,
\begin{eqnarray}
  |\F''_{kQ}| &\geq&  \frac{(1-\epsilon')(1-\alpha)}{(1+\alpha)^2} \,
\uu^2\,D^-(a;I)\, (2P+1)^d / \|T\|^2 \notag \\
  &&\qquad - \ K\big( (2P+1)^d-(2(P-Q)+1)^d\big)\notag \\
  &\geq&   \frac{(1-\epsilon')(1-\alpha)}{(1+\alpha)^2} \, \uu^2\,D^-(a;I)\,
(2P+1)^d / \|T\|^2 \notag \\
  &&\qquad - \  \alpha \frac{(1-\epsilon')(1-\alpha)}{(1+\alpha)^2} \,
\uu^2\,D^-(a;I)(2P+1)^d / \|T\|\notag \\
  &\geq&   \frac{(1-\epsilon')(1-\alpha)^2}{(1+\alpha)^2} \,
\uu^2\,D^-(a;I)\, (2P+1)^d / \|T\|^2 .
\end{eqnarray}

\noindent {\bf Claim 1}: If $J_k=\{j\in I: f_{jR}\in \F''_{kQ}\}$,
$J= \bigcup_{k\in (2P+1) \mathbb Z^d} J_k$ and  $$\F_Q(J)
=\bigcup_{k\in (2P+1) \mathbb Z^d}\F''_{kQ},$$ then
$\F_{Q}(J)$ is a Riesz sequence with lower Riesz bound
$\frac A B \cc(\epsilon')(1-\alpha)^2\uu^2$.
\vskip12pt

\noindent {\it Proof of Claim 1.} To see this, consider
$\widetilde \G_k =\{\widetilde g_{k'}:\ \|k'-k\|_\infty <P \}$
which  are disjoint subsets of $\G$. Furthermore,
(\ref{equation:ensureRiesz}) ensures that for  $k\in (2P+1) \mathbb Z^d$,
the set $\F''_{kQ}$ is a Riesz basis sequence in $ {\rm span}\, \G_k$,
where the lower Riesz  constant $\cc(\epsilon')(1-\alpha)^2\uu^2$
is given by Theorem~\ref{theorem:RestrictedInvertTheoremFinite} and does
not depend on $k$ or $P$.  For $\{a_j\}_{j\in J}\in \ell_2(J)$ we have, using
Lemma~\ref{lemma:PartitionRiesz},
\begin{eqnarray*}
  \Big\|\sum a_j f_{jQ}\Big\|^2&=&\Big\|\sum_{k \in (2P+1) \mathbb Z^d}
\sum_{j \in J_k} a_j f_{jQ}   \Big\|^2
        \geq    \frac {1/B} {1/A} \sum_{k\in (2P+1) \mathbb Z^d} \|
 \sum_{j \in J_k} a_j f_{jQ}   \|^2 \notag \\
         &\geq& \frac {A} {B} \sum_{k\in (2P+1) \mathbb Z^d}
\cc(\epsilon')(1-\alpha)^2\uu^2 \|  \{a_j\}_{j\in J_k}   \|^2 \notag \\
          &=& \frac {A\cc(\epsilon')}{B} (1-\alpha)^2\uu^2  \|
\{a_j\}_{j\in J}   \|^2.
\end{eqnarray*}
We conclude that $\F_{Q}(J)$ is a Riesz sequence with lower Riesz
bound
\[\frac {A}{B}\cc(\epsilon')(1-\alpha)^2 \uu^2,\]
 so Claim 1 is
shown.

It remains to show that we can replace $\F_{Q}(J)$ by
$\F(J)= \{f_j,\ j\in J \}= \{f_i,\ f_{iQ}\in  \F_Q(J) \}$,
 while controlling the lower Riesz bound. For $\{a_j\}_{j\in J}$ we have
\begin{eqnarray*}
  \|\sum a_j f_j \|&\geq & \|\sum a_j f_{jQ} \|- \|\sum a_j (f_j-f_{jQ})
\| \notag\\
    &\geq&  \Big(\frac{A} B \cc(\epsilon')\Big)^{\frac 1 2}\,(1-\alpha)\uu \| \{a_j\} \| -
\|(L_I-L_{IQ})^\ast \{a_i\} \|  \notag \\
    &\geq&  \Big(\frac{A}B\cc(\epsilon') (1-\alpha)^2 \uu^2\Big)^{\frac 1 2}\ \| \{a_j\} \| -
\Big(\frac {A \delta\,\cc(\epsilon') \uu^2\,} {8B}\Big)^{\frac 1 2} \| \{a_i\} \|  \notag \\
    &\geq& \Big(\frac A B \cc(\epsilon')\uu^2 ((1-\alpha)^2 - \frac \delta 8 )\Big)^{\frac 1 2}
\|\{a_i\}\|\notag\\
    &\geq& \Big(\frac A B \cc(\epsilon')\uu^2 ((1-2 \alpha - \frac \delta 8 )\Big)^{\frac 1 2}
\|\{a_i\}\|\notag\\
    &\geq&\Big( \frac A B \cc(\epsilon')\uu^2 ((1-\frac \delta 4 - \frac \delta 4 )\Big)^{\frac 1 2}
 \|\{a_i\}\|
    \geq \Big( (1-\delta) \cc(\epsilon)\ \, \uu^2\,\frac {A} {B} \Big)^{\frac 1 2}\|\{a_i\}\|.
\end{eqnarray*}

Clearly,
\begin{eqnarray}
 D(a;J)&=& D^-(a;J) \notag \\ &\geq&
\frac{(1-\epsilon')(1-\alpha)^2}{(1+\alpha)^2}\,
\uu^2\,\frac{ D^-(a;I)}{ \|T\|^2} \geq  (1-\epsilon) \,
\uu^2\frac{D^-(a;I)}{ \|T\|^2}. \notag
\end{eqnarray}
\hfill $\square$

\subsection{Proof of Theorem ~~\ref{theorem:MainForRiesz}, assuming (B)}

The only arguments in the proof of Theorem~\ref{theorem:MainForRiesz},
assuming (A), that require adjustments are Claim 1 and
the subsequent computations.

Let $K,P,\epsilon,\epsilon', \alpha$ be given as in the proof of Theorem~\ref{theorem:MainResult}, assuming (A). Choose $Q$ as in \eqref{equation:operatornomLR} with $\frac A B$ replaced by $1$.
Let $r'\in \ell_1(\mathbb Z^d)$  with $|\langle \widetilde g_n, \widetilde g_{n'}\rangle| \leq r'(n-n')$. Let $B'$
be the optimal Bessel bound of $\{g_n\}$.
Choose $R'>0$ so that
\begin{eqnarray}
  \Delta_{r'}(R')\frac{\|T\|^2 B'}{\cc(\epsilon') \uu^2}  K^2(2Q+1)^{2d} <\frac \delta 8.\notag
\end{eqnarray}

Set $W=2P+R'$.
Similarly to \eqref{equation:DropBorderRegions1}, we increase $P$ so that
  \begin{eqnarray*}\label{equation:DropBorderRegions2}
    K\big( W^d-(2(P-Q)+1)^d\big)\leq \alpha \uu^2 \frac{(1-\epsilon')(1-\alpha)}{(1+\alpha)^2} D^-(a;I)(2P+1)^d / \|T\|,
  \end{eqnarray*}
  while maintaining $W=2P+R'$.

Define $J$ and $J_k$, $k\in\mathbb Z^d$ as done in the proof of Theorem~\ref{theorem:MainResult}, assuming (A).
Let $x_k=\sum_{j\in J_k} a_j f_{jQ}$, $k\in\mathbb Z^d$, and $x=\sum x_k$.

Then
\begin{eqnarray*}\sum_{k} \sum_{k'\neq k} &&\langle x_k, x_{k'}\rangle \\
  &=&\sum_{k} \sum_{k'\neq k} \sum_{j\in J_k}\sum_{j'\in J_{k'}} \sum_{n: \|a(j)-n\|\leq Q}\sum_{n': \|a(j')-n'\|\leq Q} a_j \overline{a}_{j'} \langle f_j, g_n \rangle \overline{\langle f_{j'}, g_{n'} \rangle}\langle \widetilde g_n, \widetilde g_{n'} \rangle \notag \\
  &=&\sum_{k} \sum_{k'\neq k} \sum_{j\in J_k}\sum_{j'\in J_{k'}} \sum_{n: \|a(j)-n\|\leq Q}\sum_{n': \|a(j')-n'\|\leq Q} a_j \overline{a}_{j'} \langle f_j, g_n \rangle \overline{\langle f_{j'}, g_{n'} \rangle}M^W_{n,n'}\notag \\
  &=&\sum_{n} \sum_{n'}  \sum_{j\in J: \|a(j)-n\|\leq Q}\sum_{j'\in J: \|a(j')-n'\|\leq Q} a_j \overline{a}_{j'} \langle f_j, g_n \rangle \overline{\langle f_{j'}, g_{n'} \rangle}M^{R'}_{n,n'}\notag \\
  &=& \langle  S ,M^{R'} S \rangle,
\end{eqnarray*}
where
$$
    S=\big\{ \sum_{j\in J: \|a(j)-n\|\leq Q} a_j \langle f_j, g_n \rangle \big\}_n.
$$

We now compute the norm of $S$.
\begin{eqnarray}
  \|S\|_{\ell_2(\mathbb Z^d)}^2&=&\sum_n \big| \sum_{j\in J: \|a(j)-n\|\leq Q} a_j \langle f_j, g_n \rangle \big|^2
  \leq \sum_n \big| \sum_{j\in J: \|a(j)-n\|\leq Q} |a_j| \|f_j\| \|g_n\| \big|^2\notag \\
  &\leq&\|T\|^2 B'  \sum_n \big| \sum_{j\in J: \|a(j)-n\|\leq Q} |a_j| \big|^2\notag \\
  &\leq&\|T\|^2 B'  K(2Q+1)^d \sum_n  \sum_{j\in J: \|a(j)-n\|\leq Q} |a_j|^2\notag \\
  &\leq&\|T\|^2 B'  K(2Q+1)^d \sum_j  \sum_{n: \|a(j)-n\|\leq Q} |a_j|^2\notag \\
  &\leq&\|T\|^2 B'  K(2Q+1)^{2d} \sum_j  |a_j|^2\notag\notag \\
  &\leq&\|T\|^2 B'  K(2Q+1)^{2d} \sum_k \sum_{j\in J_k}  |a_j|^2\notag \\
  &\leq&\|T\|^2 B'     K(2Q+1)^{2d}  \sum_k \frac 1 {\cc(\epsilon') \uu^2} \|x_k\|^2\notag\\
  &=&\frac{\|T\|^2 B'}{\cc(\epsilon') \uu^2}  K(2Q+1)^{2d}  \sum_k \|x_k\|^2\notag.
\end{eqnarray}
Here, we used that for each $n$ at most $K(2Q+1)^d$ indices $j$ satisfy $\|a(j)-n\|_\infty \leq Q$, and, for each $j$ there are at most $(2Q+1)^d$ indices $n$ with $\|a(j)-n\|_\infty\leq Q$.
Recall that $B'$ is the Bessel bound of $\{g_n\}$ which therefore bounds $\{\|g_n\|\}_n$.

We conclude that for $x_k=\sum_{j\in J_k} a_j f_{jQ}$, $k\in\mathbb Z^d$,
  \begin{eqnarray*}
    \big| \sum_{k\neq k'} \langle x_k, x_{k'} \rangle \big| &\leq& \|S\|\,  \|M^{R'}\| \, \|S\| \notag \\
    &\leq& \frac{\|T\|^2 B'}{\cc(\epsilon')\uu^2}  K(2Q+1)^{2d}  \big(\sum_k \|x_k\|^2\big) K\Delta_{r'}(R')
    \notag \\
    &\leq & \frac \delta {8} \sum_k \|x_k\|^2.
  \end{eqnarray*}
  Now,
  \begin{eqnarray*}
  \|\sum_{j\in J} a_j f_{jQ}\|^2&=&
    \| \sum x_k \|^2 =  \sum_k\sum_{k'}\langle x_k,x_{k'}  \rangle \notag \\
            &=& \sum_k\|x_k\|^2 +\sum_k\sum_{k'\neq k}\langle x_k,x_{k'}  \rangle \geq  \sum_k\|x_k\|^2 -\sum_k\sum_{k'\neq k}|\langle x_k,x_{k'}  \rangle |\notag \\
            &\geq& \sum_k\|x_k\|^2 -  \frac \delta {8} \sum_k\|x_k\|^2\geq \cc(\epsilon') \uu^2 (1-\alpha)^2(1-\frac \delta {8})  \sum_{j\in J}|a_{j}|^2
  \end{eqnarray*}

 For $\F(J)= \{f_j,\ j\in J \}= \{f_i,\ f_{iQ}\in j\in \F_Q(J) \}$ and  $\{a_j\}\in \ell_2(J)$ we
 compute
\begin{eqnarray*}
  \|\sum a_j f_j \|&\geq & \|\sum a_j f_{jQ} \|- \|\sum a_j (f_j-f_{jQ}) \| \notag\\
    &\geq&\Big( \cc(\epsilon') \uu^2 (1-\alpha)^2(1-\tfrac \delta {8}) \Big)^{\frac 1 2}  \| \{a_j\} \| - \|(L_I-L_{IQ})^\ast \{a_j\} \|  \notag \\
    &\geq& \Big(\cc(\epsilon') \uu^2 (1-\alpha)^2(1-\tfrac \delta {8})  - \frac{\delta \cc(\epsilon')\uu^2}{8} \Big)^{\frac 1 2} \|\{a_j\}\|\notag\\
    &\geq&\cc(\epsilon')^{\frac 1 2}\Big((1 - \frac \delta {8} )^3- \tfrac \delta 8 \Big)^{\frac 1 2} \uu \|\{a_j\}\|\notag\\
    &\geq&\cc(\epsilon')^{\frac 1 2}\Big(1 - 3\tfrac \delta {8} - \tfrac \delta 8\Big)^{\frac 1 2} \uu\|\{a_j\}\|\notag\\
    &\geq&\cc(\epsilon')^{\frac 1 2}(1 - \tfrac \delta 2)^{\frac 1 2}\uu \|\{a_j\}\|\geq \Big((1-\delta) \cc(\epsilon) \uu^2\Big)^{\frac 1 2}\|\{a_j\}\|\notag .
\end{eqnarray*}
\hfill $\square$

\section{Gabor molecules and the Proof of Theorem~\ref{theorem:GaborExample1}}
\label{section:Gabor}

Similarly to the notion  {\bf Gabor system} $(\varphi;\Lambda)$ in Section~\ref{Section:Notation}, we define a {\bf Gabor multi-system} $(\varphi^1,\varphi^2,\ldots,\varphi^n;\Lambda^1,\Lambda^2, \ldots ,
\Lambda^n)$ generated by $n$ functions and $n$ sets of time frequency shifts as the
union of the corresponding Gabor systems
\[ (\varphi^1;\Lambda^1) \cup (\varphi^2;\Lambda^2)\cup \cdots \cup (\varphi^n;\Lambda^n).\]
Recall that the {\bf short-time Fourier transform} of a tempered distribution $f\in S'(\R^d)$ with respect to a Gaussian window function $g_0 \in S(\R^d)$ is
\[ V_{g_0}f (x,\omega) = \langle f,\pi(x,\omega)g_0 \rangle= \langle f,M_\omega T_x g_0 \rangle,\ \ \mbox{for}\ \
\lambda=(x,w)\in \R^{2d}.\]
A {\bf system of  Gabor molecules} $\{\varphi_\lambda\}_{\lambda\in\Lambda}$ associated to an enveloping function $\Gamma:\R^{2d}\rightarrow \R$ and a set of time frequency shifts $\Lambda\subseteq \R^{2d}$ consists of elements whose short-time Fourier transform have a common envelope of concentration:
\begin{eqnarray*}
|V_{g_0}\varphi_{x,\omega}(y,\xi)| \le \Gamma(y-x,\xi-\omega), \text{ for all } \lambda= (x,\omega)\in \Lambda,
\  (y,\xi)\in \R^{2d}.
\end{eqnarray*}

For $1\le p \le \infty$, the {\bf modulation space} $M^p(\R^d)$ consists of all tempered
distributions $f\in S'(\R^d)$ such that
\begin{equation}\label{E1}
 \|f\|_{M^p} = \|V_{g_0}f\|_{L^p} = \left ( \int \int_{\R^{2d}}|
\langle f,M_\omega T_x g_0 \rangle |^p
dx\ dw \right ) ^{1/p} < \infty
\end{equation}
with the usual adjustment for $p=\infty$.
It is known \cite{G} that $M^p$ is a Banach space for all
$1\le p\le \infty$, and any non-zero
function $g\in M^1$ can be substituted for the Gaussian $g_0$ in
(\ref{E1})  to define an equivalent
norm for $M^p$.
It is known
(see \cite{BCHZ2} Theorem 8 (a)) that in case $(\varphi,\Lambda)$ is a frame, $\varphi\in S_0(\R^d)$, then $(\varphi,\Lambda)$ is $\ell_1$--self-localized.

Theorem~\ref{theorem:GaborExample1} is  a special case of the following,
more general result.

\begin{theorem}\label{theorem:GaborExample2} Let $\epsilon,\delta>0$.
Let $\{g_\lambda\}_{\lambda \in \Lambda}\subseteq S_0(\R^d)$ be a set of
$\ell_1$--self-localized Gabor molecules with $\|g_\lambda\|\geq \uu$ and
Bessel bound $B<\infty$. Then exists a set $\Lambda_{\epsilon\delta}
\subseteq \Lambda$  so that\\[.1cm]

\noindent (1) \quad $\displaystyle  \frac{D(\Lambda_{\epsilon\delta})}
{D^-(\Lambda)} \geq \frac{(1-\epsilon)} B\uu^2$\\[.1cm]

\noindent (2) \quad $\{g_\lambda\}_{\lambda\in \Lambda_{\epsilon\delta}}$
is a Riesz sequence with lower Riesz bound $\cc(\epsilon)(1-\delta)\uu^2$.
\end{theorem}

\begin{proof} 
  Set
  \begin{eqnarray}
    a:\Lambda \longrightarrow \mathbb Z^{2d}, \quad \lambda \mapsto
{\rm arg}\min_{n\in\mathbb Z^{2d}} \|\lambda - \tfrac 1 2 n \|_\infty.  \notag
\end{eqnarray}
Now, $D^-(a,\mathbb Z^{2d})=2^{-2d} D^-(\Lambda)$.
 Choose $g\in S_0(\R^d)$ with $\G=(g,\frac 1 2 \mathbb Z^{2d})
=\{\pi(\frac 1 2 n) g\}$ being a tight frame.  As $g\in S_0(\R^d)$,
we have $(g,\frac 1 2 \mathbb Z^{2d})$ is $\ell_1$-self-localized and
$(\{\varphi_\lambda\}_{\lambda},a,(g,\frac 1 2 \mathbb Z^{2d}))$ is $\ell_1$-localized
\cite{BCHZ2}.

A direct application of Theorem~\ref{theorem:MainResult}, assuming (B),
guarantees for each $\epsilon,\delta >0$ the existence of $\Lambda_{\epsilon\delta}
\subseteq \Lambda$ with $\{\varphi_\lambda\}_{\lambda\in \Lambda_{\epsilon\delta}}$ is a Riesz sequence and
\begin{eqnarray}
  D^-(\Lambda_\epsilon)= 2^{2d} D(a;\Lambda_\epsilon)
\ge2^{2d}\frac{(1-\epsilon)}{B}D^{-}(a;I)=
\frac{(1-\epsilon)}{B}D^{-}(\Lambda). \label{equation:MainResultGabor}
\end{eqnarray}

\end{proof}

\section{Appendix}

We will need a minor extension of  Theorem
\ref{theorem:RestrictedInvertTheoremFinite1}.
 Its proof is based on the formulation of
Casazza \cite{C} and Vershynin \cite{V1}.

\begin{theorem}[\bf Restricted Invertibility Theorem]
\label{theorem:RestrictedInvertTheoremFinite}
There exists a continuous and monotone function
$\cc:(0,1)\longrightarrow (0,1)$ so that for every $n\in \N$ and every linear
operator $T:\ell_2^n \rightarrow \ell_2^n$ with $\|Te_i\|\geq\uu$ for
$i=1,2,\ldots,n$ and $\{e_i\}_{i=1}^n$ an orthonormal basis for $\ell_2^n$,
there is a subset $J_\epsilon\subseteq \{1,2,\ldots,n\}$ satisfying\\[.1cm]

\noindent (1)  \quad $\displaystyle \frac {|J_\epsilon|} n \ge
\frac{(1-\epsilon)\uu^2}{\|T\|^2}, $ and\\[.1cm]

\noindent (2) \quad $\displaystyle \|\sum_{j\in J_\epsilon}b_jTe_j\|^2
\ge \cc(\epsilon)\,\uu^2\, \sum_{j\in J_\epsilon}|b_j|^2,\quad
\{b_j\}_{j\in J}\in \ell_2(J_\epsilon)$.
\end{theorem}

\begin{proof}  Theorem \ref{theorem:RestrictedInvertTheoremFinite1}
does not assert
continuity
of $\cc$. Due to the defining property of $\cc$, we can choose  $\cc$ monotone,
 and replacing $\cc$ with $\cc_\zeta$, $\zeta>0$, with
$\cc_\zeta (\epsilon)=\int_{\min\{0,\epsilon-\zeta\}}^\epsilon
\cc(\epsilon')\,d\epsilon'$ ensures continuity of $\cc$.

Next, we want to replace the traditional assumption $\|Te_i\|=1$
with  $\|Te_i\|\geq \uu$.
Given $T$, define an operator $S$ by
\[ Se_i = \frac{Te_i}{\|Te_i\|},\quad i=1,2,\ldots, n.\]
Now,
\begin{eqnarray*}
\Big\|\sum_{i=1}^na_iSe_i\Big\| &=& \Big\|\sum_{i=1}^n \frac{a_i}{\|Te_i\|}
Te_i\Big\|
\le \|T\|\left ( \sum_{i=1}^n \left | \frac{a_i}{\|Te_i\|}
\right |^2 \right )^{1/2}
\le \frac{\|T\|}{\uu}\left ( \sum_{i=1}^n |a_i|^2 \right )^{1/2}.
\end{eqnarray*}
Hence,
\[ \|S\| \le \frac{\|T\|}{\uu}.\]
Applying Theorem  \ref{theorem:RestrictedInvertTheoremFinite1}
 to the operator $S$
with $\|Se_i\|=1$, $i=1,\ldots,n$, we obtain $J\subseteq \{1,\ldots,n\}$ with
\vskip12pt

\noindent (1) \quad $\displaystyle |J|\ge \frac{(1-\epsilon)\uu^2}{\|T\|^2}n,$
\vskip12pt

\noindent (2) \quad $\displaystyle \big\|\sum_{j\in J} b_jTe_j\big\|^2
=  \big\|\sum_{j\in J} b_j\|Te_j\| Se_j\big\|^2\ge
\cc(\epsilon)^2\sum_{j\in J} |b_j|^2\|Te_j\|^2\ge
\cc(\epsilon)^2\uu^2 \sum_{j\in J} |b_j|^2.
$
\end{proof}

We will also need a simple inequality for Riesz sequences.

\begin{lemma}\label{lemma:PartitionRiesz}
Let $\{f_i\}_{i\in I}$ be a Riesz basis sequence with bounds $A,B$.  Then for
any partition $\{I_j\}_{j\in J}$ of $I$ we have for all scalars
$\{a_i\}_{i\in I}$,
\[ \frac{A}{B}\sum_{j\in J}\|\sum_{i\in I_j}a_if_i\|^2 \le \|
\sum_{i\in I}a_if_i\|^2\le \frac{B}{A}\sum_{j\in J}
\|\sum_{i\in I_j}a_if_i\|^2.\]
\end{lemma}

\begin{proof}
\begin{eqnarray*}
\frac{A}{B}\sum_{j\in J}\|\sum_{i\in I_j}a_if_i\|^2&\le&
\frac{A}{B}B\sum_{j\in J}
\sum_{i\in I_j}|a_i|^2 = A\sum_{i\in I}|a_i|^2\le\|\sum_{i\in I}a_if_i\|^2\\
&\le& B \sum_{i\in I}|a_i|^2
= B\sum_{j\in J}\sum_{i\in I_j}|a_i|^2
\le \frac{B}{A}\sum_{j\in J}\|\sum_{i\in I_j}a_if_i\|^2
\end{eqnarray*}
\end{proof}

\end{document}